\documentclass[12pt,a4paper]{amsart}
\allowdisplaybreaks[1]

\DeclareFontEncoding{OT2}{}{}
\DeclareFontSubstitution{OT2}{cmr}{m}{n}

\DeclareFontFamily{OT2}{cmr}{\hyphenchar\font45 }
\DeclareFontShape{OT2}{cmr}{m}{n}{%
   <5><6><7><8><9>gen*wncyr%
   <10><10.95><12><14.4><17.28><20.74><24.88>wncyr10}{}

\DeclareMathAlphabet{\mathcyr}{OT2}{cmr}{m}{n}
\DeclareMathAlphabet{\mathcyb}{OT2}{cmr}{b}{n}
\SetMathAlphabet{\mathcyr}{bold}{OT2}{cmr}{b}{n}

\usepackage{amstext}
\usepackage{amsthm}
\usepackage{amssymb}

\newtheorem{thm}{Theorem}[section]

\newtheorem{prop}[thm]{Proposition}
\newtheorem{cor}[thm]{Corollary}

\theoremstyle{definition}

\theoremstyle{remark}
\newtheorem{rem}[thm]{Remark}

\newcommand{\sh}{\mathbin{\mathcyr{sh}}}

\begin{document}

\title[Derivation relation for FMZVs in $\widehat{\mathcal{A}}$]{Derivation relation for finite multiple zeta values in $\widehat{\mathcal{A}}$}

\author{Hideki Murahara}
\address[Hideki Murahara]{Nakamura Gakuen University Graduate School, 5-7-1, Befu, Jonan-ku,
Fukuoka, 814-0198, Japan}
\email{hmurahara@nakamura-u.ac.jp}

\author{Tomokazu Onozuka}
\address[Tomokazu Onozuka]{Multiple Zeta Research Center, Kyushu University 744, Motooka, Nishi-ku,
Fukuoka, 819-0395, Japan}
\email{t-onozuka@math.kyushu-u.ac.jp}


\begin{abstract}
 Ihara, Kaneko, and Zagier proved the derivation relation for multiple zeta values. 
 The first named author obtained its counterpart for finite multiple zeta values in $\mathcal{A}$. 
 In this paper, we present its generalization in $\widehat{\mathcal{A}}$.   
\end{abstract}

\maketitle

\section{Introduction}
 For $k_1,\dots,k_r\in \mathbb{Z}_{\ge1}$ with $k_r \ge 2$, the multiple zeta values (MZVs) is defined by 
\begin{align*}
 \zeta(k_1,\dots, k_r)=\sum_{1\le n_1<\cdots <n_r} \frac {1}{n_1^{k_1}\cdots n_r^{k_r}}.
\end{align*} 
For $n\in\mathbb{Z}_{\ge1}$, we define the $\mathbb{Q}$-algebra $\mathcal{A}_{n}$ by
\[
\mathcal{A}_{n}:=\biggl(\prod_{p}\mathbb{Z}/p^{n}\mathbb{Z}\biggr)\,\big/\,\biggl(\bigoplus_{p}\mathbb{Z}/p^{n}\mathbb{Z\biggr)}=\{(a_{p})_{p}\mid a_{p}\in\mathbb{Z}/p^{n}\mathbb{Z}\}/\sim,
\]
where $(a_{p})_{p}\sim(b_{p})_{p}$ are identified if and only if
$a_{p}=b_{p}$ for all but finitely many primes $p$. For $k_{1},\dots,k_{r}\in\mathbb{Z}_{\ge1}$
and $n\in\mathbb{Z}_{\ge1}$, the finite multiple zeta values (FMZVs) in $\mathcal{A}_n$ is defined by 
\begin{align*}
\zeta_{\mathcal{A}_{n}}(k_{1},\dots,k_{r}) & :=\biggl(\sum_{1\le n_{1}<\cdots<n_{r}\le p-1}\frac{1}{n_{1}^{k_{1}}\cdots n_{r}^{k_{r}}}\bmod p^{n}\biggr)_{p}\in\mathcal{A}_{n}.
\end{align*}

Recently, Rosen \cite{Ros15} introduced the $\mathbb{Q}$-algebra $\widehat{\mathcal{A}}$. 
By natural projections $\mathcal{A}_{n}\to\mathcal{A}_{n-1}$, we define $\widehat{\mathcal{A}}:=\lim_{\gets n}\mathcal{A}_{n}$, where we put the discrete topology on each $\mathcal{A}_{n}$. 
We also define the natural projections $\pi\colon \prod_p\mathbb{Z}_p\to \widehat{\mathcal{A}}$ and $\pi_n\colon \widehat{\mathcal{A}}\to \mathcal{A}_n$ for each $n$, where $\mathbb{Z}_p$ is the ring of $p$-adic integers. 
Then FMZVs in  $\widehat{\mathcal{A}}$ is given by
\begin{align*}
 \zeta_{\widehat{\mathcal{A}}}(k_{1},\dots,k_{r})
 &:=\pi\biggl(\biggl(\sum_{1\le n_{1}<\cdots<n_{r}\le p-1}\frac{1}{n_{1}^{k_{1}}\cdots n_{r}^{k_{r}}}\biggr)_{p}\biggr)\in\widehat{\mathcal{A}}.
\end{align*}
We can easily check that 
$\pi_n(\zeta_{\widehat{\mathcal{A}}}(\boldsymbol{k}))=\zeta_{\mathcal{A}_n}(\boldsymbol{k})$
for $\boldsymbol{k}\in\mathbb{Z}_{\ge1}^{r}$.
Furthermore, we define $\boldsymbol{p}:=\pi((p)_p)\in\widehat{\mathcal{A}}$ and we also use the notation  $\boldsymbol{p}:=\pi_n\circ\pi((p)_p)\in\mathcal{A}_n$
(for details, see Rosen \cite{Ros15} and Seki \cite{Sek16}).

We recall Hoffman's algebraic setup with a slightly different convention (see Hoffman \cite{Hof97}).
Let $\mathfrak{H} :=\mathbb{Q} \left\langle x,y \right\rangle$ be the noncommutative polynomial ring in two indeterminates $x$, $y$, and $\mathfrak{H}^1$ (resp. $\mathfrak{H}^0$) its subring $\mathbb{Q} +y \mathfrak{H}$ (resp. $\mathbb{Q} +y \mathfrak{H} x$). 
Set $z_{k} := yx^{k-1}$ $(k\in\mathbb{Z}_{\ge1})$. 
We define the $\mathbb{Q}$-linear map $\mathit{Z} \colon \mathfrak{H}^0 \to \mathbb{R}$ by $\mathit{Z}(1):=1$, $\mathit{Z}( z_{k_1} \cdots z_{k_r}):= \zeta (k_1,\ldots, k_r)$. 

A derivation $\partial$ on $\mathfrak{H}$ is a $\mathbb{Q}$-linear map $\partial\colon\mathfrak{H}\to\mathfrak{H}$ satisfying Leibniz's rule $\partial(ww^{\prime})=\partial(w)w^{\prime}+w\partial(w^{\prime})$. 
Such a derivation is uniquely determined by its images of generators $x$ and $y$. 
Set $z:=x+y$. 
For each $l\in\mathbb{Z}_{\ge1}$, the derivation $\partial_l$ on $\mathfrak{H}$ is defined by 
$\partial_l(x):=yz^{l-1}x$ and $\partial_l(y):=-yz^{l-1}x$. 
We note that $\partial_l(1)=0$ and $\partial_l(z)=0$. 
In addition, $R_x$ is the $\mathbb{Q}$-linear map given by $R_x(w):=wx$ for any $w\in\mathfrak{H}$.
\begin{thm}[Derivation relation for MZVs; Ihara-Kaneko-Zagier {\cite[Theorem 3]{IKZ06}}] \label{der_MZVs}
For $l\in\mathbb{Z}_{\ge1}$ and $w \in \mathfrak{H}^0 $, we have 
 \begin{align*} 
 \mathit{Z} (\partial_{l}(w)) = 0.  
 \end{align*}
\end{thm}

Similar to the definition of $\mathit{Z}$, 
we define two $\mathbb{Q}$-linear maps $\mathit{Z}_{\mathcal{A}_n} \colon\mathfrak{H}^1 \to \mathcal{A}_n$ and  $\mathit{Z}_{\widehat{\mathcal{A}}} \colon\mathfrak{H}^1 \to \widehat{\mathcal{A}}$
by $\mathit{Z}_{\mathcal{A}_n} (1)=\mathit{Z}_{\widehat{\mathcal{A}}} (1):=1$, $\mathit{Z}_{\mathcal{A}_n} (z_{k_1} \cdots z_{k_r}):=\zeta_{\mathcal{A}_n} (k_1,\ldots ,k_r)$, and $\mathit{Z}_{\widehat{\mathcal{A}}} (z_{k_1} \cdots z_{k_r}):=\zeta_{\widehat{\mathcal{A}}} (k_1,\ldots ,k_r)$. 
We write $\mathcal{A}:=\mathcal{A}_1$.
The derivation relation for FMZVs in  $\mathcal{A}$ was conjectured by Oyama and proved by the first named author. 
\begin{thm}[Derivation relation for FMZVs in $\mathcal{A}$; Murahara {\cite[Theorem 2.1]{Mur16}}]  \label{derF}
 For $l\in\mathbb{Z}_{\ge1}$ and $w \in y\mathfrak{H}x $, we have 
 \begin{align*} 
  \mathit{Z}_\mathcal{A} (R_x^{-1}\partial_{l}(w)) = 0
 \end{align*}
 in the ring $\mathcal{A}$.
\end{thm} 
In this paper, we prove a generalization of the above theorem in the ring $\widehat{\mathcal{A}}$. 
For non-negative integers $m$ and $n$, we define $\beta_{m,n}\colon \mathbb{Q}\langle\langle x,y \rangle\rangle[[u,v]] \rightarrow\frak{H}$ by setting $\beta_{m,n}(w)$ to be the coefficients of $u^mv^n$ in $w$. 
\begin{thm}[Main theorem] \label{main}
 For $m\in\mathbb{Z}_{\ge0}$ and  $w\in y\mathfrak{H}$, we have
 \begin{align*}
  &\sum_{n=0}^\infty \mathit{Z}_{\widehat{\mathcal{A}}}
  \left(\beta_{m,n}R_{x}^{-1} \Delta_u R_{x} \left(w-wyu\frac{1}{1+xu} \cdot
  \frac{xv}{1-xv} \right)\right)\boldsymbol{p}^n\\
  &=\mathit{Z}_{\widehat{\mathcal{A}}}(w)\mathit{Z}_{\widehat{\mathcal{A}}}\left(\beta_{m,0}\left(\frac{1}{1-yu}\right)\right)
 \end{align*}
 in the ring $\widehat{\mathcal{A}}$, where $\Delta_u$ is an automorphism on $\mathbb{Q}\langle\langle x,y \rangle\rangle[[u,v]]$ 
 given by
 \[
  \Delta_u:=\exp \left( \sum_{l=1}^\infty \frac{\partial_l}{l}(-u)^l \right).
 \]
\end{thm}
\begin{rem}
Since $\mathit{Z}_{\mathcal{A}}(1,\dots,1)=0$ (see, for example, Hoffman \cite[eq.(15)]{Hof15}), 
Theorem \ref{main} is a generalization of the equality
\[
 \mathit{Z}_{\mathcal{A}}
 \left(\beta_{m,0}R_{x}^{-1} \Delta_u R_{x} (w) \right)
 =0
\]
for $m\in\mathbb{Z}_{\ge0}$, which was obtained by Ihara (see Horikawa-Murahara-Oyama \cite[Section 5.3]{HMO18}). 
We note that this is equivalent to Theorem \ref{derF}.
\end{rem}
As a corollary of our main theorem, we have Hoffman's relation (see Hoffman \cite[Theorem5.1]{Hof92} for original formula) for FMZVs in
$\widehat{\mathcal{A}}$.
\begin{cor} \label{HofF}
 For $w\in y\mathfrak{H}$, we have
 \[
  \mathit{Z}_{\widehat{\mathcal{A}}}
  \left(R_{x}^{-1} \partial_{1} (wx)\right)
  =-\sum_{n=1}^{\infty}
   \mathit{Z}_{\widehat{\mathcal{A}}}
   \left(wyx^n\right)\boldsymbol{p}^n
   -\mathit{Z}_{\widehat{\mathcal{A}}} (w)
   \mathit{Z}_{\widehat{\mathcal{A}}} (y)
 \]
 in the ring $\widehat{\mathcal{A}}$.
\end{cor}

\section{Proof of the main theorem}
\subsection{Notation}
The harmonic product $\ast$ and the shuffle product $\sh$ on $\mathfrak{H}^{1}$ are defined by
\begin{align*}
 1\ast w&=w\ast 1:= w, \\
 z_{k} w_{1} \ast z_{l} w_{2}&:= z_{k} (w_{1} \ast z_{l} w_{2}) + z_{l} (z_{k} w_{1} \ast w_{2}) + z_{k+l} (w_{1} \ast w_{2}),  \\
 1 \sh w&= w\sh 1:= w , \\
 u w_{1} \sh v w_{2}&:= u (w_{1} \sh v w_{2}) + v (u w_{1} \sh w_{2})
\end{align*}
($k,l \in \mathbb{Z}_{\ge1}$, $u,v\in\{x,y\}$ and $w$, $w_{1}$, $w_{2}$ are words in $\mathfrak{H}^{1}$), together with $\mathbb{Q}$-bilinearity. 
The harmonic product $\ast$ and the shuffle product $\sh$ are commutative and associative, therefore $\mathfrak{H}^{1}$ is a $\mathbb{Q}$-commutative algebra with respect to $\ast$ and $\sh$, respectively  
(see Hoffman \cite{Hof97}).

\subsection{Propositions and lemmas}
In this subsection, we prepare some propositions which will be used later. 
\begin{prop} \label{harmonic}
 For $w_1,\,w_2\in\mathfrak{H}^1$, we have
 \begin{align*}
  \mathit{Z}_{\widehat{\mathcal{A}}}(w_1\ast w_2)
  &=\mathit{Z}_{\widehat{\mathcal{A}}}(w_1) \mathit{Z}_{\widehat{\mathcal{A}}}(w_2)
 \end{align*}
 in the ring $\widehat{\mathcal{A}}$.
\end{prop}
\begin{proof}
 This is obtained by the definition of the harmonic product. 
\end{proof}
\begin{prop}[Jarossay {\cite{Jar17}}, Seki {\cite[Theorem 6.4]{SekDoc}}] \label{shuffle}
For $w_1,\,w_2\in\mathfrak{H}^1$ with $w_2=z_{k_1}\cdots z_{k_r}$, we have
\begin{align*}
&\mathit{Z}_{\widehat{\mathcal{A}}}(w_1\sh w_2)\\
&=(-1)^{k_1+\cdots+k_r}\sum_{l_1,\ldots,l_r\in\mathbb{Z}_{\ge0}}
 \left[\prod_{i=1}^r\binom{k_i+l_i-1}{l_i}\right]
 \mathit{Z}_{\widehat{\mathcal{A}}}(w_1z_{k_r+l_r}\cdots z_{k_1+l_1}) \boldsymbol{p}^{l_1+\cdots+l_r}
\end{align*}
in the ring $\widehat{\mathcal{A}}$.
\end{prop}
\begin{prop}[Ihara-Kaneko-Zagier {\cite[Corollary 3]{IKZ06}}] \label{IKZ}
For $w\in\mathfrak{H}^{1}$, we have
\begin{align*}
\frac{1}{1-yu} \ast w=\frac{1}{1-yu} \sh\Delta_u(w).
\end{align*}
\end{prop}

\subsection{Proof of Theorem \ref{main}}
By Proposition \ref{harmonic}, we have
\begin{align} \label{lhs}
 \mathit{Z}_{\widehat{\mathcal{A}}}
  \left(\beta_{m,0} \left( \frac{1}{1-yu} \ast w\right)\right)
 =\mathit{Z}_{\widehat{\mathcal{A}}}
  \left(\beta_{m,0} \left(\frac{1}{1-yu}\right)\right)
  \mathit{Z}_{\widehat{\mathcal{A}}}(w).
\end{align}
On the other hand, we have
\[
 \mathit{Z}_{\widehat{\mathcal{A}}}(w\sh y^r)
 =\sum_{n=0}^\infty
  \mathit{Z}_{\widehat{\mathcal{A}}}
  \left(\beta_{0,n} \left( w\left( -y\frac{1}{1-xv} \right)^r \right)\right)\boldsymbol{p}^n
\]
holds by Proposition \ref{shuffle}. 
Then we find
\begin{align*}
 \mathit{Z}_{\widehat{\mathcal{A}}}
  \left(\beta_{m,0}\left( \frac{1}{1-yu} \sh\Delta_u(w) \right)\right)
 &=\mathit{Z}_{\widehat{\mathcal{A}}}
  \left(\beta_{m,0}\left( \Delta_u(w) \sh \frac{1}{1-yu} \right)\right) \\
 &=\mathit{Z}_{\widehat{\mathcal{A}}}
  \left(\beta_{m,0}\left( \Delta_u(w) \sh \sum_{i=0}^\infty y^iu^i \right)\right) \\
 &=\sum_{n=0}^\infty
  \mathit{Z}_{\widehat{\mathcal{A}}}
  \left(\beta_{m,n}\left( \sum_{i=0}^\infty \Delta_u(w) \left( -yu\frac{1}{1-xv} \right)^i \right)\right)\boldsymbol{p}^n \\
 &=\sum_{n=0}^\infty
  \mathit{Z}_{\widehat{\mathcal{A}}}
  \left(\beta_{m,n} \left( \Delta_u(w) \frac{1}{1+yu(1-xv)^{-1}} \right)\right)\boldsymbol{p}^n.
\end{align*}
From the direct calculation, we have
\begin{align*}
 &\{1+yu(1-xv)^{-1}\}^{-1}\\
 &=\{(1-xv+yu)(1-xv)^{-1}\}^{-1}\\
 &=(1-xv)(1-xv+yu)^{-1}\\
 &=(1-xv)\{(1+yu)(1-(1+yu)^{-1}xv)\}^{-1}\\
 &=(1-xv)(1-(1+yu)^{-1}xv)^{-1}(1+yu)^{-1}\\
 &=(1-xv)\sum_{i=0}^\infty\left(\frac{1}{1+yu}xv\right)^i(1+yu)^{-1}.
\end{align*}
Since $\Delta_u(x)=(1+yu)^{-1}x$, we have
\begin{align*}
 &\{1+yu(1-xv)^{-1}\}^{-1}\\
 &=(1-xv)
  \sum_{i=0}^\infty\left(\Delta_u(x)v\right)^i(1+yu)^{-1}\\
 &=(1+yu)^{-1}
  +\sum_{i=1}^\infty\left(\Delta_u(x)-x\right)\left(\Delta_u(x)\right)^{i-1}v^i(1+yu)^{-1}.
\end{align*} 
Since $x=\Delta_u(x+y(1+xu)^{-1}xu)$, we have
\begin{align*}
 &\{1+yu(1-xv)^{-1}\}^{-1}\\ 
 &=(1+yu)^{-1}-\sum_{i=1}^\infty\Delta_u\left(y\frac{xu}{1+xu}\right)\Delta_u\left(x^{i-1}\right)v^i(1+yu)^{-1}\\
 &=R_{x}^{-1}\Delta_uR_{x}\left(1-yu\frac{1}{1+xu}\cdot\frac{xv}{1-xv}\right).
\end{align*}
Hence we get 
\begin{align} \label{rhs}
 \begin{split}
  &\mathit{Z}_{\widehat{\mathcal{A}}}
   \left(\beta_{m,0} \left( \frac{1}{1-yu} \sh\Delta_u(w) \right)\right)\\
  &=\sum_{n=0}^\infty \mathit{Z}_{\widehat{\mathcal{A}}}
   \left(\beta_{m,n}R_{x}^{-1} \Delta_u R_{x} \left(w-wyu\frac{1}{1+xu} \cdot
   \frac{xv}{1-xv} \right)\right)\boldsymbol{p}^n.
 \end{split}
\end{align}
By (\ref{lhs}), (\ref{rhs}), and Proposition \ref{IKZ}, we finally obtain the theorem.

\subsection{Proof of Corollary \ref{HofF}}
When $m = 1$ in Theorem \ref{main}, we have
\begin{align*}
 &\sum_{n=0}^\infty \mathit{Z}_{\widehat{\mathcal{A}}}
 \left( \beta_{1,n}R_{x}^{-1} \Delta_u R_{x} \left(w-wyu\frac{1}{1+xu} \cdot
 \frac{xv}{1-xv} \right)\right)\boldsymbol{p}^n \\
 &=\mathit{Z}_{\widehat{\mathcal{A}}}(w)\mathit{Z}_{\widehat{\mathcal{A}}}(y). 
\end{align*}
Since 
\begin{align*} 
 &\beta_{1,n}R_{x}^{-1} \Delta_u R_{x} \left(w-wyu\frac{1}{1+xu} \cdot
  \frac{xv}{1-xv} \right) \\
 &=
  \begin{cases}
   -wyx^n &\textrm{if } n\ge1, \\
   -R_{x}^{-1} \partial_1 R_{x} (w) &\textrm{if } n=0,
  \end{cases}
\end{align*}
we get the result.

\section*{Acknowledgement}
The authors would like to express their sincere gratitude to 
Doctor Minoru Hirose for valuable comments.


\end{document}